\documentclass{amsart}[12pt]
\def\doctype{}

\usepackage{graphicx,subcaption}
\usepackage{latexsym,amssymb}
\usepackage{color}
\usepackage{tikz}
\usetikzlibrary{decorations.markings}
\usepackage{fancyhdr}
\usepackage{hyperref}

\newcommand{\comment}[1]{}

\fancypagestyle{titlepage}{
\fancyhead[R]{\doctype}
\fancyhead[CL]{}
\cfoot{\vspace{5pt} \thepage}
}

\let\oldsection\section
\newcommand\boldsection[1]{\oldsection{\bf #1}}
\newcommand\starsection[1]{\oldsection*{\bf #1}}
\makeatletter
\renewcommand\section{\@ifstar\starsection\boldsection}
\makeatother

\newtheoremstyle{theorem}
  {12pt}		  % space above
  {0pt}  % space below
  {\sl}  % bofy font
  {\parindent}     % ident - empty=no indent,  \parindent= paragraph indent
  {\bf}  % thm head font
  {. }    % punctuation after thm head
  { }    % space after thm head: `` ``=normal \newline=linebreak
  {}     % thm head specification
\theoremstyle{theorem}
\newtheorem{thm}{Theorem}[section]  % 1st argument is your name for it
\newtheorem{prop}[thm]{Proposition}

\newtheoremstyle{definition}
  {12pt}		  % space above
  {0pt}  % space below
  {}  % bofy font
  {\parindent}     % ident - empty=no indent,  \parindent= paragraph indent
  {\bf}  % thm head font
  {. }    % punctuation after thm head
  { }    % space after thm head: `` ``=normal \newline=linebreak
  {}     % thm head specification
\theoremstyle{definition}

\newtheorem{ex}[thm]{Example}

\renewcommand{\proofname}{Proof}

\makeatletter
\renewenvironment{proof}[1][\proofname]{\par
  \pushQED{\qed}%
  \normalfont \partopsep=\z@skip \topsep=\z@skip
  \trivlist
  \item[\hskip\labelsep
        \scshape
    #1\@addpunct{.}]\ignorespaces
}{%
  \popQED\endtrivlist\@endpefalse
}
\makeatother

\makeatletter
\renewcommand*\@maketitle{%
  \normalfont\normalsize
  \@adminfootnotes
  \@mkboth{\@nx\shortauthors}{\@nx\shorttitle}%
  \global\topskip42\p@\relax % 5.5pc   "   "   "     "     "
  \@settitle
  \ifx\@empty\authors \else {\vskip 1em
\vtop{\centering\shortauthors\@@par}} \fi
  \ifx\@empty\@date \else {\vskip 1em \vtop{\centering\@date\@@par}}\fi % MY CHANGE
  \ifx\@empty\@dedicatory
  \else
    \baselineskip18\p@
    \vtop{\centering{\footnotesize\itshape\@dedicatory\@@par}%
      \global\dimen@i\prevdepth}\prevdepth\dimen@i
  \fi
  \@setabstract
  \normalsize
  \if@titlepage
    \newpage
  \else
    \dimen@34\p@ \advance\dimen@-\baselineskip
    \vskip\dimen@\relax
  \fi
} % end \@maketitle
\renewcommand*\@adminfootnotes{%
  \let\@makefnmark\relax  \let\@thefnmark\relax
  \ifx\@empty\@subjclass\else \@footnotetext{\@setsubjclass}\fi
  \ifx\@empty\@keywords\else \@footnotetext{\@setkeywords}\fi
  \ifx\@empty\thankses\else \@footnotetext{%
    \def\par{\let\par\@par}\@setthanks}%
  \fi
\thispagestyle{titlepage}
}
\makeatother

\begin{document}

\title{\large Incidence structures near configurations of type $(n_3)$}

\author{Peter J.~Dukes}
\address{\rm Peter J.~Dukes:
Mathematics and Statistics,
University of Victoria, Victoria, Canada
}
\email{dukes@uvic.ca}

\author{Kaoruko Iwasaki}
\address{\rm Kaoruko Iwasaki:
Mathematics and Statistics,
University of Victoria, Victoria, Canada
}
\email{mphysics.f@gmail.com}

\thanks{Research of the first author is supported by NSERC grant number 312595--2017}

\date{\today}

\begin{abstract}
An $(n_3)$ configuration is an incidence structure equivalent to a linear hypergraph on $n$ vertices which is both 3-regular and 3-uniform.  We investigate a variant in which one constraint, say 3-regularity, is present, and we allow exactly one line to have size four, exactly one line to have size two, and all other lines to have size three.  In particular, we study planar (Euclidean or projective) representations, settling the existence question and adapting Steinitz' theorem for this setting.
\end{abstract}

\maketitle
\hrule

%\tableofcontents \hrule

%\setpagewiselinenumbers
%\modulolinenumbers[5]
%\linenumbers

\section{Introduction}

A \emph{geometric} $(n_r,m_k)$ \emph{configuration} is a set of $n$ points and $m$ lines in the Euclidean (or projective) plane such that every line contains exactly $k$ points (uniformity) and every point is incident with exactly $r$ lines (regularity).  Gr\"{u}nbaum's book \cite{Gbook} is an excellent reference for the major results on this topic, including many variations.  

Our study here is a variation in which uniformity (or, dually, regularity) is mildly relaxed; see also \cite{BP-quasi} and \cite[\S 6.8]{PSbook} for prior investigations in this direction.

An \emph{incidence structure} is a triple $(P,\mathcal{L},\iota)$, where $P$ is a set of \emph{points}, $\mathcal{L}$ is a set of \emph{lines}, and $\iota \subseteq P \times \mathcal{L}$ is a relation called \emph{incidence}.  Here, we assume $P$ is finite, say $|P|=n$, and no two different lines are incident with the same set of points.  In this case, we may identify $\mathcal{L}$ with a set system on $P$. Equivalently, the points and lines can be regarded as vertices and edges, respectively, of a (finite) hypergraph.  
A \emph{point-line incidence structure} is an incidence structure in which any two different points are incident with at most one line.  (Under this assumption, the associated hypergraph is called \emph{linear}.)  We make the standard assumption in this setting that each line is incident with at least two points.

A point-line incidence structure in which all lines have the same number of points and all points are incident with the same number of lines is called a (combinatorial) \emph{configuration}.  As before, the abbreviation $(n_r,m_k)$ represents that there are $n$ points, each incident with exactly $r$ lines, and $m$ lines, each incident with exactly $k$ points.
In alternate language, such an incidence structure is then an $r$-regular, $k$-uniform linear hypergraph on $n$ vertices and $m$ edges.   Counting incidences in two ways, one has $nr=mk$.
In either the geometric or combinatorial setting, notation in the case $n=m$ (hence $k=r$) is simplified to $(n_k)$ \emph{configuration}.

{\bf Examples.}
The complete graph $K_4$ is a $(4_3,6_2)$ configuration, also known as a `quadrangle'.  The Desargues' configuration is a $(10_3)$ configuration.  The  Fano plane is a $(7_3)$ configuration.  The quadrangle and Desargues' configuration are geometric, whereas the Fano plane is not.

A point-line incidence structure is \emph{geometric} if it can be realized with points and lines in the plane, with the usual notion of incidence.
%When a certain point-line incidence structure with specified parameters (in particular, a combinatorial configuration) is to be realized in this way, there is typically an additional requirement %that geometric incidences may not be `thrown away'.
%The \emph{type} of a configuration, whether abstract or geometric, is the abbreviation $(n_r, m_k)$, or simply $(n_r)$ in the case that $n=m$ and $k=r$.

Consider a point-line incidence structure with $n_i$ points of degree $r_i$, $i=1,\dots,s$, and $m_j$ lines of size $k_j$, $j=1,\dots,t$.  As before, counting incidences yields the relation
\begin{equation}
\label{incidences}
\sum_{i=1}^s n_i r_i = \sum_{j=1}^t m_j k_j.
\end{equation}
The \emph{signature} of such a point-line incidence structure is the pair of polynomials
$(f(x),g(y))$, where $f(x)=\sum_{i=1}^s n_i x^{r_i}$ and $g(y)=\sum_{j=1}^t m_j y^{k_j}$.
Note that (\ref{incidences}) can be rewritten as $f'(1)=g'(1)$.

%With a reluctance to extend the usage of the term `configuration', we refer to such a generalization as a \emph{pseudo-configuration}.  
%For its \emph{type}, we use a list of point counts $n_i$ subscripted by degrees $r_i$, followed by line counts $m_j$ subscripted by their sizes $k_j$.  (It is best to avoid writing this in %general to save the reader from parsing two types of subscripts!)

In particular, an $(n_r,m_k)$ configuration has signature $(nx^r,my^k)$.
As an example with mixed line sizes, deleting one point from the Fano plane results in a point-line incidence structure with signature $(6x^3,3y^2+ 4y^3)$.  For another example, placing an extra point on one line of the Desargues' configuration results in a point-line incidence structure with signature $(x+10x^3,4y+9y^3)$.

In \cite{BP-quasi}, Borowski and Pilaud use point-line incidence structures with signature $(ax^3+bx^4,cy^3+dy^4)$ to produce the first known examples of $(n_4)$ configurations for $n=37$ and $43$.  This serves as excellent motivation for further study of point-line incidence structures which are `approximately' configurations.

In this note, we consider structures with signature $(nx^3,y^2+(n-2)y^3+y^4)=(nx^3,ny^3+y^2(1-y)^2)$, and in particular their geometric realizations.  One may view these as approximate $(n_3)$ configurations. Our main result is point-line incidence structures with such signatures exist if and only if $n \ge 9$, and in fact there is one with a geometric representation in each case.  This is shown in Section~2.  On the other hand, not all structures with our signature are geometric.  We observe that Steinitz' theorem on  representing $(n_3)$ configurations with at most one curved line assumes a stronger form in our setting, essentially characterizing which of our structures are realizable via their bipartite incidence graph.  This is covered in Section~3.  We conclude with a few extra remarks, including a connection to `fuzzy' configurations, in which some points are replaced by intervals.

\section{Existence}

Our first result easily settles the existence question for the abstract combinatorial case.

\begin{thm}
\label{comb-existence}
There exists a $($combinatorial$)$ point-line incidence structure having signature $(nx^3,ny^3+y^2(1-y)^2)$ if and only if $n \ge 9$.  
\end{thm}

\begin{proof}
Suppose there exists such an incidence structure.  Consider the line $L=\{p_1,p_2,p_3,p_4\}$ of size four. Each point $p_i$ is incident with exactly two other lines. These $8$ additional lines are distinct, by linearity and the fact that $p_i, p_j$ are already together on $L$ for each $1 \le i < j \le 4$.  It follows that the configuration has at least 9 lines, and so $n \ge 9$.

Conversely, suppose $n \ge 9$.  Let $\mathcal{L}$ denote the set of translates of $\{0,1,3\}$ modulo $n-1$.  Then $(\{0,1,\dots,n-2\},\mathcal{L})$ is an $((n-1)_3)$ configuration.  Add a new point $\infty$ and take the adjusted family of lines
$$\mathcal{L}':=\mathcal{L} \setminus \{\{0,1,3\},\{4,5,7\}\} \cup \{\{0,1,3,\infty\}, \{4,\infty\},\{5,7,\infty\}\}.$$
It is simple to check that $(\{0,1,\dots,n-2,\infty\},\mathcal{L}')$ is a point-line incidence structure with the desired signature.
\end{proof}

We turn now to the geometric case. The line adjustments used in the proof of Theorem~\ref{comb-existence} cannot be applied to an arbitrary geometric $((n-1)_3)$ configuration.  Moreover, the Mobi{\"u}s-Kantor $(8_3)$ configuration does not admit a geometric realization.  To this end, we start with an example on $9$ points.

\begin{ex}
\label{example9}
Figure~\ref{fig:example9} shows a (geometric) point-line incidence structure with signature $(9x^3,y^2+ 7y^3+y^4)$ given as (a) a projective realization with one point at infinity and (b) a Euclidean realization.  The lines of size 2 and 4 are highlighted. The dual incidence structure, in which there is exactly one point of degree 2 and one point of degree 4, is drawn in (c).
\end{ex}

\tikzset{->-/.style={decoration={
  markings,
  mark=at position #1 with {\arrow{>}}},postaction={decorate}}}
 
 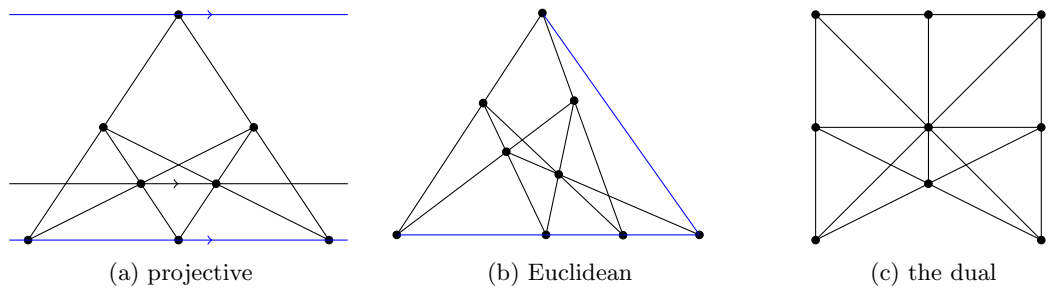
\begin{figure}[htbp]
    \centering
    \begin{subfigure}[b]{0.3\textwidth}
\begin{tikzpicture}[xscale=0.5,yscale=0.5]
\filldraw (0,0) circle [radius=.1cm];
\filldraw (4,0) circle [radius=.1cm];
\filldraw (-4,0) circle [radius=.1cm];
\filldraw (0,6) circle [radius=.1cm];
\filldraw (2,3) circle [radius=.1cm];
\filldraw (-2,3) circle [radius=.1cm];
\filldraw (1,1.5) circle [radius=.1cm];
\filldraw (-1,1.5) circle [radius=.1cm];
\draw [blue,->-=.6](-4.5,0)--(4.5,0);
\draw [blue,->-=.6](-4.5,6)--(4.5,6);
\draw [->-=.5] (-4.5,1.5)--(4.5,1.5);
\draw (-4,0)--(0,6);
\draw (4,0)--(0,6);
\draw (0,0)--(2,3);
\draw (0,0)--(-2,3);
\draw (-4,0)--(2,3);
\draw (4,0)--(-2,3);
\end{tikzpicture}
        \caption{projective}
    \end{subfigure}
\quad
    \begin{subfigure}[b]{0.3\textwidth}
\begin{tikzpicture}[xscale=0.4,yscale=0.4]
\filldraw (1.94,-2.78) node(a){} circle[radius=.125cm];
\filldraw (12.,-2.78) node(b){} circle[radius=.125cm];
\filldraw (6.78,4.6) node(c){} circle[radius=.125cm];
\filldraw (6.900080628268113,-2.779721990569654) node(d){} circle[radius=.125cm];
\filldraw (9.460040156039415,-2.789900756087632) node(e){} circle[radius=.125cm];
\filldraw (7.836720335992288,1.6862158940697083) node(f){} circle[radius=.125cm];
\filldraw (5.581856842862181,-0.01398512121645012) node(g){} circle[radius=.125cm];
\filldraw (4.810313241399388,1.6047738546899784) node(h){} circle[radius=.125cm];
\filldraw (7.321738309475747,-0.7692402681687687) node(i){} circle[radius=.125cm];
\draw [blue](a.center)--(b.center);
\draw [blue](b.center)--(c.center);
\draw (a.center)--(c.center);
\draw (c.center)--(e.center);
\draw (h.center)--(e.center);
\draw (h.center)--(d.center);
\draw (a.center)--(f.center);
\draw (b.center)--(g.center);
\draw (d.center)--(f.center);
\end{tikzpicture}
        \caption{Euclidean}
    \end{subfigure}
\hfill
    \begin{subfigure}[b]{0.3\textwidth}
\begin{tikzpicture}[scale=0.5]
\filldraw (0,1.5) circle [radius=.1cm];
\filldraw (0,3) circle [radius=.1cm];
\filldraw (0,6) circle [radius=.1cm];
\filldraw (3,0) circle [radius=.1cm];
\filldraw (3,3) circle [radius=.1cm];
\filldraw (3,6) circle [radius=.1cm];
\filldraw (-3,0) circle [radius=.1cm];
\filldraw (-3,3) circle [radius=.1cm];
\filldraw (-3,6) circle [radius=.1cm];
\draw (-3,0)--(-3,6);
\draw (3,0)--(3,6);
\draw (0,1.5)--(0,6);
\draw (-3,6)--(3,6);
\draw (-3,3)--(3,3);
\draw (-3,0)--(3,6);
\draw (-3,6)--(3,0);
\draw (-3,0)--(3,3);
\draw (-3,3)--(3,0);
\end{tikzpicture}
        \caption{the dual \phantom{XXXXX}}
    \end{subfigure}
    \caption{a point-line incidence structure with signature $(9x^3,y^2+ 7y^3+y^4)$}
    \label{fig:example9}
\end{figure}

Here, it is helpful to cite a result which is  used (very mildly) below, and (more crucially) for our structural considerations in Section 3.  This result, due to Steinitz, roughly says that configurations of type $(n_3)$ are `nearly' geometric. However, the reader is encouraged to see Gr\"unbaum's discussion \cite[\S 2.6]{Gbook} of `unwanted incidences', which in certain cases cannot be avoided.

\begin{thm}[Steinitz,\cite{Steinitz}; see also \cite{Gbook}]
For every combinatorial $(n_3)$ configuration, there is a representation of all but at most one of its incidences by points and lines in the plane.
\end{thm}

This result, in combination with with Example~\ref{example9} and the existence of combinatorial $(n_3)$ configurations for $n \ge 7$, allows us to adapt the argument in Theorem~\ref{comb-existence} to a geometric one.

\begin{thm}
\label{geom-existence}
For each $n \ge 9$, there exists a point-line incidence structure having signature $(nx^3,ny^3+y^2(1-y)^2)$ with incidences represented as points and lines in the plane.  
\end{thm}

\begin{proof}
First, we consider values $n \ge 14$.  Take disjoint configurations $(P,\mathcal{L})$ and $(Q,\mathcal{H})$ of types $(7_3)$ and $((n-7)_3)$, respectively.  Apply Steinitz' theorem to each.  Assume that $\{p_1,p_2,p_3\} \in \mathcal{L}$, yet line $p_1p_2$ in the first Steinitz embedding does not contain $p_3$.  Similarly, suppose 
$\{q_1,q_2,q_3\} \in \mathcal{H}$ where $q_3$ need not be placed on line $q_1q_2$.  If it is not, we align the drawings in the plane such that line $p_1p_2$ coincides with line $q_1q_2$, creating a line of size four.  We finish by including the additional line $p_3 q_3$.  If the latter configuration has a geometric realization (which we may assume for $n \ge 16$),
we can simply place the drawings so that $p_3$ is on lines $q_1q_2q_3$, and include $p_1p_2$ as an additional line.  In either case, we have a geometric point-line configuration with the desired signature.

\begin{figure}[htbp]
\begin{tikzpicture}[scale=0.5]
\filldraw (0,0) node(a){} circle[radius=.1cm];
\filldraw (1,1) node(b){} circle[radius=.1cm];
\filldraw (1,-1) node(c){} circle[radius=.1cm];
\filldraw (2,2) node(d){} circle[radius=.1cm];
\filldraw (2,-2) node(e){} circle[radius=.1cm];
\filldraw (4,1) node(f){} circle[radius=.1cm];
\filldraw (4,-1) node(g){} circle[radius=.1cm];
\filldraw (2.5,0) node(h){} circle[radius=.1cm];
\filldraw (6,0) node(i){} circle[radius=.1cm];
\draw [blue,->-=.6](a.center)--(i.center);
\draw (a.center)--(d.center);
\draw (a.center)--(e.center);
\draw (i.center)--(d.center);
\draw (i.center)--(e.center);
\draw [->-=.6](b.center)--(f.center);
\draw [->-=.6](c.center)--(g.center);
\draw (b.center)--(g.center);
\draw (c.center)--(f.center);
\draw [blue](d.center)--(e.center);
\end{tikzpicture}
\quad
\begin{tikzpicture}[scale=0.5]
\filldraw (0,0) node(a){} circle[radius=.1cm];
\filldraw (1,2) node(b){} circle[radius=.1cm];
\filldraw (0,4) node(c){} circle[radius=.1cm];
\filldraw (2,0) node(d){} circle[radius=.1cm];
\filldraw (2,2) node(e){} circle[radius=.1cm];
\filldraw (3,0) node(f){} circle[radius=.1cm];
\filldraw (2,4) node(g){} circle[radius=.1cm];
\filldraw (3,2) node(h){} circle[radius=.1cm];
\filldraw (3,4) node(i){} circle[radius=.1cm];
\filldraw (4.5,0) node(j){} circle[radius=.1cm];
\filldraw (0,6) node(k){} circle[radius=.1cm];

\draw[blue] (a.center)--(j.center);
\draw (a.center)--(k.center);
\draw (a.center)--(g.center);
\draw (b.center)--(h.center);
\draw (c.center)--(i.center);
\draw (c.center)--(d.center);
\draw (g.center)--(d.center);;
\draw (f.center)--(i.center);
\draw (f.center)--(k.center);
\draw (j.center)--(k.center);
\draw [blue](i.center)--(j.center);
\end{tikzpicture}
\quad
\begin{tikzpicture}
\filldraw (-1,0) node(a){} circle [radius=.05cm];
\filldraw (-1,1) node(b){} circle [radius=.05cm];
\filldraw (0,1) node(c){} circle [radius=.05cm];
\filldraw (1,1) node(d){} circle [radius=.05cm];
\filldraw (1,0) node(e){} circle [radius=.05cm];
\filldraw (1,-1) node(f){} circle [radius=.05cm];
\filldraw (0,-1) node(g){} circle [radius=.05cm];
\filldraw (-1,-1) node(h){} circle [radius=.05cm];
\draw [blue](c.center)--(g.center);
\draw (b.center)--(h.center);
\draw  [->-=.7] (a.center)--(d.center);
\draw  [->-=.3] (f.center)--(a.center);
\draw  [->-=.7] (e.center)--(b.center);
\draw  [->-=.3] (h.center)--(e.center);
\draw  [->-=.3] (h.center)--(c.center);
\draw  [->-=.3] (f.center)--(c.center);
\draw (d.center)--(f.center);
\draw  [->-=.7] (g.center)--(b.center);
\draw  [->-=.7] (g.center)--(d.center);
\end{tikzpicture}
\quad
\begin{tikzpicture}[scale=0.5]
\filldraw (-2.7,0) node(a){} circle [radius=.1cm];
\filldraw (2.7,0) node(b){} circle [radius=.1cm];
\filldraw (-2,2) node(c){} circle [radius=.1cm];
\filldraw (-2,-2) node(d){} circle [radius=.1cm];
\filldraw (2,2) node(e){} circle [radius=.1cm];
\filldraw (2,-2) node(f){} circle [radius=.1cm];
\filldraw (-1,0) node(g){} circle [radius=.1cm];
\filldraw (1,0) node(h){} circle [radius=.1cm];
\filldraw (-.5,1) node(i){} circle [radius=.1cm];
\filldraw (.5,-1) node(j){} circle [radius=.1cm];
\draw [blue](a.center)--(b.center);
\draw [->-=.5](a.center)--(c.center);
\draw [->-=.5] (d.center)--(a.center);
\draw [->-=.5] (b.center)--(e.center);
\draw [->-=.5] (f.center)--(b.center);
\draw (c.center)--(h.center);
\draw (d.center)--(i.center);
\draw (e.center)--(j.center);
\draw (g.center)--(f.center);
\draw [->-=.6] (d.center)--(c.center);
\draw [->-=.6] (f.center)--(e.center);
\draw [blue] (i.center)--(j.center);
\end{tikzpicture}
    \caption{examples for $n=10,\dots,13$}
    \label{fig:others}
\end{figure}
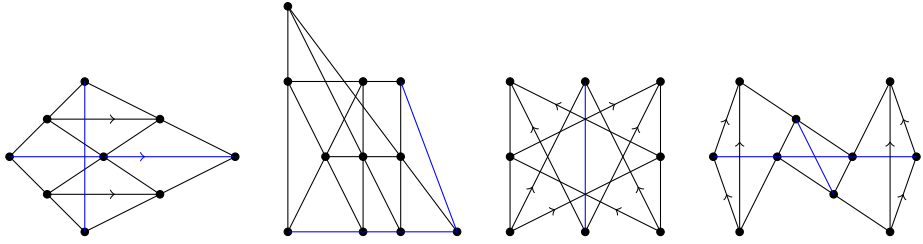

For the cases $9 \le n \le 13$, we use direct examples.  The case $n=9$ is shown in Example~\ref{example9}; the other values are shown in Figure~\ref{fig:others}.  Note that projective points are used in some cases, as indicated with arrows on parallel (Eudlidean) lines.
\end{proof}

\begin{ex}
As an illustration of the proof technique, Figure~\ref{fig:example14} shows how two Fano planes with missing incidences can be aligned to produce an example in the case $n=14$. 
\end{ex}

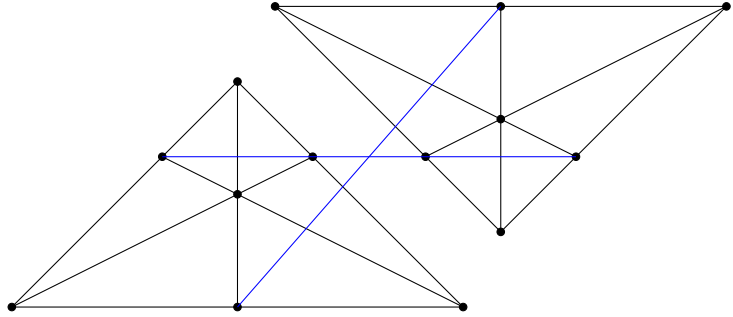
\begin{figure}[htbp]
\begin{tikzpicture}[scale=0.5]
\filldraw (0,0) node(a){} circle [radius=.1cm];
\filldraw (-6,0) node(b){} circle [radius=.1cm];
\filldraw (6,0) node(c){} circle [radius=.1cm];
\filldraw (0,6) node(d){} circle [radius=.1cm];
\filldraw (-2,4) node(e){} circle [radius=.1cm];
\filldraw (2,4) node(f){} circle [radius=.1cm];
\filldraw (0,3) node(g){} circle [radius=.1cm];
\draw (b.center)--(c.center);
\draw (a.center)--(d.center);
\draw (b.center)--(d.center);
\draw (c.center)--(d.center);
\draw (b.center)--(f.center);
\draw (c.center)--(e.center);

\filldraw (7,8) node(aa){} circle [radius=.1cm];
\filldraw (1,8) node(bb){} circle [radius=.1cm];
\filldraw (13,8) node(cc){} circle [radius=.1cm];
\filldraw (7,2) node(dd){} circle [radius=.1cm];
\filldraw (5,4) node(ee){} circle [radius=.1cm];
\filldraw (9,4) node(ff){} circle [radius=.1cm];
\filldraw (7,5) node(gg){} circle [radius=.1cm];
\draw (bb.center)--(cc.center);
\draw (aa.center)--(dd.center);
\draw (bb.center)--(dd.center);
\draw (cc.center)--(dd.center);
\draw (bb.center)--(ff.center);
\draw (cc.center)--(ee.center);

\draw [blue](e.center)--(ff.center);
\draw [blue](a.center)--(aa.center);

\end{tikzpicture}
    \caption{a compound example with $n=14$ points}
    \label{fig:example14}
\end{figure}

\section{Structure}

In Theorem~\ref{geom-existence}, we merely proved the existence of one incidence structure with signature $(nx^3,ny^3+y^2(1-y)^2)$ having a geometric representation. It is easy to build (combinatorial) such incidence structures which are non-geometric using a similar method.

\begin{prop}
For each $n \ge 14$, there exists a point-line incidence structure having signature $(nx^3,ny^3+y^2(1-y)^2)$
and with no planar representation.
\end{prop}

\begin{proof}
Take disjoint configurations $(P,\mathcal{L})$ and $(Q,\mathcal{H})$ of types $(7_3)$ and $((n-7)_3)$, respectively.  Consider the configuration $(P \cup Q,\mathcal{K})$, where $\mathcal{K}$ is formed from the lines of $\mathcal{L} \cup \mathcal{H}$ by replacing lines $\{p_1,p_2,p_3\} \in \mathcal{L}$ and $\{q_1,q_2,q_3\} \in \mathcal{H}$ with
$\{p_1,p_2,p_3,q_4\}$ and $\{q_2,q_3\}$.  This incidence structure has the desired signature.  It is non-geometric, since its `restriction' to $P$ is the Fano plane, which has no realization in the plane.
\end{proof}

Given an incidence structure $(P,\mathcal{L},\iota)$, its \emph{Levi graph} is the bipartite graph with vertex partition $(P,\mathcal{L})$ and edge set $\{\{p,L\}: (p,L) \in \iota\}$.  It is easy to see that the Levi graph of a linear incidence structure has girth at least six.  Steinitz' argument for obtaining geometric representations essentially works by iteratively removing vertices of lowest degree in the Levi graph, and then drawing the corresponding objects (as needed) in reverse.

We can use the Levi graph to test when a combinatorial point-line incidence structure of our signature has a geometric representation.

\begin{thm}
\label{struct}
A point-line incidence structure with signature $(nx^3,ny^3+y^2(1-y)^2)$ has a geometric realization if and only if its Levi graph contains no cut-edge whose removal leaves a $3$-regular component whose corresponding configuration is non-geometric.
\end{thm}

\begin{proof}
Necessity of the condition is obvious. 

For sufficiency, consider the Levi graph of a structure of the given type. Remove the unique vertex of degree 2, and iteratively remove a vertex of degree less than 3 until there are no more such vertices.  Either we succeed in eliminating all vertices of the graph, or some 3-regular subgraph is left over.  In the latter case, note that the last vertex removed corresponds to one of the points on the line of size four.

If all vertices are eliminated, simply reverse the list of deletions and follow Steinitz' argument.  Since vertices get `added back' with degree at most two, we obtain a sequence of instructions of one of the following types: placing a new point, drawing a new line, placing a point on an existing line, drawing a line through an existing point, placing a point on the intersection of two existing lines, and drawing a line through a pair of existing points.  Each instruction can be carried out in the plane, and we have the desired representation.

Suppose a nonempty 3-regular graph remains after removing vertices.  By our assumption, it is the Levi graph of an $(n_3)$ configuration which admits a geometric representation.  Carry out the instructions as above to reverse vertex deletions, ensuring to begin by placing a new point on the cut-vertex to create a line of size four.  This again produces a representation of incidences in the plane.
\end{proof}

As a simple consequence, we may successfully carry out Steinitz' procedure in this setting provided the vertex of degree four in the Levi graph is incident with no cut-edge.

\section{Discussion}

A general method which we have found often works for the explicit construction of geometric `approximate' configurations can be loosely described as `moving incidences'.  It is usually not possible to simply move a point onto an existing line.  But, sometimes, a configuration can be perturbed slightly to achieve this.  (In doing so, some care must be taken to avoid unwanted incidences.)  Our hope is that such a process can be systematically described, producing a wide assortment of perturbations of known configurations.  If done in sufficient generality, such a process might handle other types beyond our simple case study of $(nx^3,ny^3+y^2(1-y)^2)$.

\begin{figure}[htbp]
    \centering
    \begin{subfigure}[b]{0.4\textwidth}
\begin{tikzpicture}
\foreach \a in {0,1,...,4}
 \draw  (\a*72: 1)--(\a*72+156: 3.17);
 \foreach \a in {0,1,...,4}
 \draw  (\a*72+12: 3.17)--(\a*72+156: 3.17);
\foreach \a in {0,1,...,4}
 \filldraw (\a*72: 1) circle [radius=.05];
 \foreach \a in {0,1,...,4}
 \filldraw (\a*72+12: 3.17) circle [radius=.05];
\node at (1.2,.2) {$p$};
\end{tikzpicture}
    \end{subfigure}
\quad
    \begin{subfigure}[b]{0.45\textwidth}
\begin{tikzpicture}
\foreach \a in {0,1,...,4}
 \draw  (\a*72: 1)--(\a*72+156: 3.17);
 \foreach \a in {0,1,...,4}
 \draw  (\a*72+12: 3.17)--(\a*72+156: 3.17);
\foreach \a in {1,...,4}
 \filldraw (\a*72: 1) circle [radius=.05];
 \foreach \a in {0,1,...,4}
 \filldraw (\a*72+12: 3.17) circle [radius=.05];
\filldraw (38:0.35)--(-38:0.35)--(-34:0.42)--(34:0.42);
\node at (0.1,0) {$I$};
\end{tikzpicture}
        \end{subfigure}
    \caption{`fuzzy' modification of a $(10_3)$ configuration}
    \label{fig:fuzzy}
\end{figure}
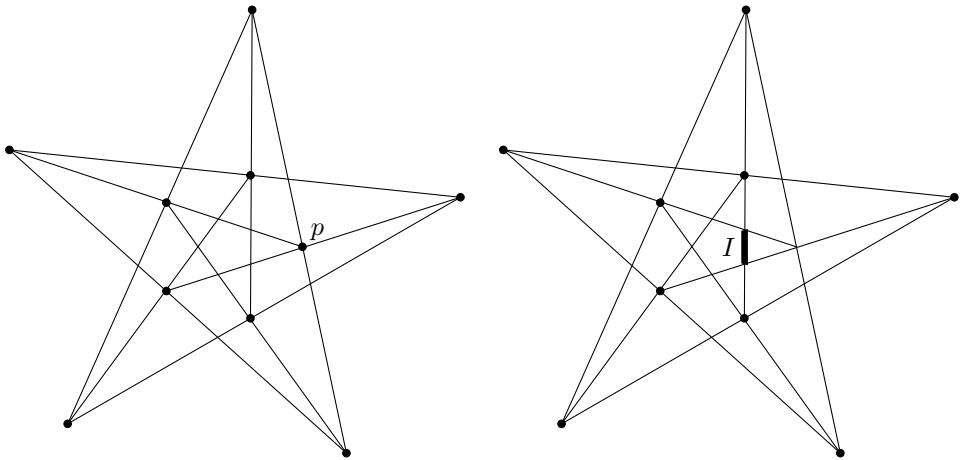

One possible way this can succeed is through the intermediate step of a `fuzzy' realization, which we loosely define as a collection of intervals (possibly points) and lines so that all incidences are represented in the usual way.  In this setting, an interval is deemed incident with a line if it intersects that line.  

Consider the example $(10_3)$ configuration shown on the left of Figure~\ref{fig:fuzzy}.  To construct the incidence structure on the right, we have removed one point, say $p$, and replaced it by an interval $I$ on two of the lines containing $p$.  With $I$ interpreted as an abstract point, the resulting incidence structure has signature $(10x^3,y^2+8y^3+y^4)$.
Under mild conditions, it is possible to perform such a replacement in a general $(n_3)$ configuration.  
It remains to shrink the interval $I$ to a point while maintaining other incidences.  This latter step is possible in the case of the fuzzy realization on the right of Figure~\ref{fig:fuzzy}, and in a few other cases we considered.  We feel an interesting question is identifying sufficient conditions for when such interval-eliminating perturbation can be carried out in general.

We have not attempted any enumeration or classification work for structures with the signature  $(nx^3,ny^3+y^2(1-y)^2)$.  But such work may offer some insight into a partial correspondence with, say, geometric $((n-1)_3)$ configurations.

\section*{Acknowledgement}

The authors would like to thank Gary MacGillivray and Vincent Pilaud for helpful discussions.

\end{document}